\documentclass[12pt]{article}
\usepackage[english]{babel}
\usepackage[latin1]{inputenc}
\usepackage[T1]{fontenc}
\usepackage{amsmath,amssymb,amsthm,xspace,enumerate,url}
\usepackage[hypertex]{hyperref}
\usepackage{svn}
\newcommand{\Ind}{
 \setbox0=\hbox{$x$}\kern\wd0\hbox to 0pt{\hss$
 \mid$\hss}\lower.9\ht0\hbox to 0pt{\hss$\smile$\hss}\kern\wd0
}
\newcommand{\indep}[3]{
 #1\mathop{\mathpalette\Ind{}}_{#2}#3
}

\newcommand{\Notind}{
 \setbox0=\hbox{$x$}\kern\wd0\hbox to 0pt{\mathchardef
 \nn=12854\hss$\nn$\kern1.4\wd0\hss}\hbox to 0pt{\hss$\mid$\hss}\lower.9\ht0
 \hbox to 0pt{\hss$\smile$\hss}\kern\wd0
}

\newcommand{\notind}[3]{
 #1\mathop{\mathpalette\Notind{}}_{#2}#3
}

\usepackage{srcltx}

\title{The free pseudospace is $n$-ample, but not $(n+1)$-ample
}

\author{Katrin Tent
}
\date{\today}

\newtheorem{satz}{Theorem}[section]
\newtheorem{theorem}{Theorem}[section]

\newtheorem{lemma}[satz]{Lemma}
\newtheorem{proposition}[satz]{Proposition}
\newtheorem{corollary}[satz]{Corollary}
\newtheorem{definition}[satz]{Definition}
\newtheorem{remark}[satz]{Remark}
\newtheorem{sideremark}[satz]{Side remark}

\newcommand{\nc}{\newcommand}
\nc{\sa}{semialgebraic\xspace}
\nc{\el}{elementary\xspace}
\nc{\low}{lower \el}
\nc{\inv}[1]{\frac{1}{#1}}
\nc{\G}{\Gamma}
\nc{\Np}{\N_{\scriptscriptstyle >0}}
\nc{\Z}{\mathbb{Z}}
\nc{\Q}{\mathbb{Q}}
\nc{\N}{\mathbb{N}}

\nc{\Rp}{\R_{\scriptscriptstyle >0}}
\nc{\C}{\mathbb{C}}
\nc{\F}{\ensuremath{\mathcal{F}}\xspace}
\nc{\K}{\mathcal{K}}
\nc{\U}{\mathbb{U}}
\nc{\fps}{free pseudospace\xspace}
\nc{\fpsn}{free pseudospace of dimension $n$\xspace}
\nc{\fpsk}{free pseudospace of dimension $n-1$\xspace}

\nc{\E}{\mathbb{E}}
\nc{\Epsilon}{{\Large $\epsilon$}} %Funktioniert nicht im Mathmode
\nc{\ap}{approximable\xspace}
\nc{\e}{\mathrm{e}}
\nc{\ii}{\,\mathrm{i}}
\nc{\Es}{\E\setminus\R_{\scriptscriptstyle\leq0}}
\renewcommand{\o}{\omega}

\DeclareMathOperator{\MR}{MR}
\DeclareMathOperator{\proj}{proj}

\DeclareMathOperator{\qftp}{qftp}
\DeclareMathOperator{\tp}{tp}
\DeclareMathOperator{\cl}{acl}
\DeclareMathOperator{\acl}{acl}
\DeclareMathOperator{\dcl}{dcl}

\DeclareMathOperator{\Cb}{Cb}

\begin{document}

\maketitle
\begin{abstract}
We give a uniform construction of free pseudospaces of dimension $n$ extending work in \cite{BP}. This
yields examples of $\omega$-stable theories which are $n$-ample, but 
not $n+1$-ample. The prime models of these theories are buildings associated
to certain right-angled Coxeter groups.

\end{abstract}

\section{Introduction}

In the investigation of geometries on strongly minimal sets the notion of
ampleness plays an important role. Algebraically closed fields are $n$-ample for all $n$ and it is not known whether there are strongly minimal sets which are
$n$-ample for all $n$ and do not interpret an infinite field. Obviously, one 
way of proving that no infinite field is interpretable in a theory is by showing that the theory is \emph{not} $n$-ample for some $n$.

In \cite{BP}, Baudisch and Pillay constructed a free pseudospace of dimension~$2$. Its theory is $\o$-stable (of infinite rank) and $2$-ample. 
F. Wagner posed the question whether this example was $3$-ample or not.

In Section~\ref{sec:construction} we give a uniform construction of a free pseudospace of dimension $n$ and
show that it is $n$-ample, but not $n+1$-ample.  It turns out that the
theory of the free pseudospace of dimension $n$ is the first order theory
of a Tits-building associated to a certain Coxeter diagram and we will investigate this connection in Section~\ref{sec:building}.

In the final section we show that there are exactly two orthogonality classes of regular types.

\section{Construction and results}\label{sec:construction}
Fix a natural number $n\geq 1$. Let $L_n$ be the language for $n+1$-coloured graphs containing predicates $V_i,i=0,\ldots n$ and an edge relation $E$.

By an $L_n$-graph  we mean an $n+1$-coloured graph with vertices of types $V_i, i=0,\ldots n$ and
an edge relation $E\subseteq \bigcup_{i=1,\ldots n}V_{i-1}\times V_i$. We say that a path  in this graph is of type $E_i$ if all its vertices are in $V_{i-1}\cup V_i$ and of type $E_i\cup \ldots \cup E_{i+j}$ if all its
vertices are in $V_{i-1}\cup\ldots\cup V_{i+j}$

The free pseudospaces will be modeled along the lines of a projective space, i.e.
we will think of vertices of type $V_i$ as $i$-dimensional spaces in a \fps. Therefore we
extend the notion of incidence as follows:
\begin{definition}\label{d:incidence} \begin{enumerate}
 \item\label{d:incidence1} We say that a vertex $x_i$ of type $V_i$ is \emph{incident} to a vertex $x_j$ of type $V_j$
if there are  vertice $x_l$ of type $V_l, l=i+1\ldots j$ such that $E(x_{l-1},x_l)$ holds.
In this case the sequence $(x_i,\ldots x_j)$ is called a \emph{dense flag}.
A \emph{flag} is a sequence of vertices $(x_1,\ldots x_k)$ in which any two vertices
are incident.

\item The \emph{residue} $R(x)$ of a vertex $x$ is the set of vertices incident with $x$.

\item We say that two vertices $x$ and $y$ \emph{intersect in the vertex} $z$ and
write $z=x\wedge y$ if the
set of vertices of type $V_0$ incident with $x$ and $y$ is exactly the set of vertices
of type $V_0$ incident with $z$. If there is no vertex of type $V_0$ incident to $x$ and $y$, we say that $x$ and $y$ intersect in the empty set.

\item We say that two vertices $x$ and $y$ \emph{generate the vertex} $z$ and write $z=x\vee y$, if the
set of vertices of type $V_n$ incident with $x$ and $y$ is exactly the set of vertices
of type $V_n$ incident with $z$. If there is no vertex of type $V_n$ incident to $x$ and $y$, we say that $x$ and $y$ generate the empty set.
\item A \emph{simple cycle} is a cycle without repetitions.
\end{enumerate}
\end{definition}

We now give an inductive definition of a \fps of dimension $n$:

\begin{definition}
 A \fps of dimension $1$ is a free pseudoplane, i.e. an $L_1$-graph which does not contain any cycles and such that any vertex has infinitely many neighbours.

Assume that a free pseudospace of dimension $n-1$ has been defined.
Then a \fps of dimension $n$ is an $L_n$-graph such that the following holds:
\begin{enumerate}
\item[$(\Sigma 1)_n$]
\begin{enumerate}
 \item The set of vertices of type $V_0\cup \ldots \cup V_{n-1}$ is a \fpsk.
\item  The set of vertices of type $V_1\cup \ldots \cup V_n$ is a free pseudospace of dimension $(n-1)$.
\end{enumerate}
\item[$(\Sigma 2)_n$]
\begin{enumerate}
 \item For any vertex $x$ of type $V_0$, $R(x)$ is a free pseudospace of dimension $(n-1)$.
 \item  For any vertex $x$ of type $V_n$, $R(x)$ is a free pseudospace of dimension $(n-1)$.
\end{enumerate}
\item[$(\Sigma 3)_n$]
\begin{enumerate}
 \item Any two vertices $x$ and $y$ intersect in some vertex $z$  or the emptyset.
 \item Any two vertices $x$ and $y$ generate  some vertex $z$ or the emptyset.
\end{enumerate}
\item[$(\Sigma 4)_n$]
\begin{enumerate}
 \item If $a$ is a vertex of type $V_n$ and $\gamma=(a,b,\ldots , b',a)$ is a simple cycle of length $k$, then there is an $E_1\cup\ldots \cup E_{n-1}$-path from $b$ to $b'$  of length at most $k-1$ in $R(a)$ all of whose $V_0$-vertices appear in $\gamma$. 
 \item If $a$ is a vertex of type $V_0$ and $\gamma=(a,b,\ldots , b',a)$ is a simple cycle of length $k$, then there is an $E_2\cup\ldots \cup E_n$-path from $b$ to $b'$  of length at most $k-1$ in $R(a)$ all of whose $V_n$-vertices appear in $\gamma$. 
\end{enumerate}

\end{enumerate}\end{definition}

Let $T_n$ denote the $L_n$-theory expressing these axioms.

Note that the inductive nature of the definition immediately has the following
consequences:

\begin{enumerate}
\item The induced subgraph on $V_j\cup\ldots\cup V_{j+m}$ is a free pseudospace
of dimension $m$. 
\item If $a$ is a vertex of type $V_{j+m}$ and 
$\gamma=(a,b,\ldots , b',a)$ is a simple cycle of length $k$ contained in  $V_j\cup\ldots\cup V_{j+m}$, then there is an $E_{j+1}\cup\ldots \cup E_{j+m-2}$-path from $b$ to $b'$  of length at most $k-1$ in $R(a)$ all of whose $V_j$-vertices appear in $\gamma$.

\item The notion of a \fpsn is \emph{self-dual}: if we put $W_i=V_{n-i}, i=0,\ldots n$, then $W_0,\ldots W_n$  with the same set of edges is again a \fpsn.

\end{enumerate}

Our first goal is to show that $T_n$ is consistent and complete.

\begin{definition} Let $A$ be a finite $L_n$-graph. The following extensions are called elementary strong extensions of $A$:
\begin{enumerate}
\item add a vertex of any type to $A$ which is connected
to at most one vertex of $A$ of an appropriate type.
\item If $(x,y,z)$ is a dense flag in $A$, add a vertex of the same  type as $y$ to $A$ which is connected to both $x$ and $z$. 
\end{enumerate}

We write $A\leq B$ if  $B$ arises from $A$  by finitely many elementary strong extensions.
\end{definition}

\begin{definition}
Let $\K_n$ be the class of finite $L_n$-graphs $A$ such that the following holds
\begin{enumerate}
\item $A$ does not contain any $E_i$-cycles for $i=1,\ldots n$.

\item If $a\neq a'$ are in $A$, they intersect in a vertex of $A$ or the emptyset.

\item If $a\neq a'$ are in $A$, they generate a vertex of $A$ or the emptyset.

\item If $(b,a,b')$ is a path with $a\in V_i,b,b'\in V_{i-1}$, and $\gamma=(a,b,\ldots , b',a)$ is an $E_i\cup E_{i-j}$ path of length $k$, then there is some $E_{i-1}\cup \ldots E_{i-j}$-path from $b$ to $b'$  of length at most $k-1$ in $R(a)$ with all $V_{i-j}$-vertices occurring in $\gamma$. 

\item If $(b,a,b')$ is a path with $a\in V_i,b,b'\in V_{i+1}$, and $\gamma=(a,b,\ldots , b',a)$ is an $E_i\cup\ldots\cup E_{i+j}$ path of length $k$, then there is some $E_{i+1}\cup\ldots \cup E_{i+j}$-path from $b$ to $b'$  of length at most $k-1$ in $R(a)$ with all $V_{i+j}$-vertices occurring in $\gamma$. 

\end{enumerate}

\end{definition}

Note that $\K_n$ is closed under finite substructures.
We next show that $(\K,\leq)$ has the amalgamation property for strong extensions.

For any finite $L_n$-graphs $A\subseteq B,C$ we denote by $B\otimes_AC$ the
\emph{free amalgam} of $B$ and $C$ over $A$, i.e. the graph on $B\cup C$ containing no edges between elements of
$B\setminus A$ and $C\setminus A$.

\begin{lemma} If  $A\leq B,C$ are in $\K_n$, then $D:=B\otimes_AC\in\K$ and
$B,C\leq D$.
\end{lemma}

\begin{proof}

Clearly, $B,C\leq D$. 
To see that $D\in\K_n$, note that if $B\in \K_n$ and $B'$ is an elementary strong
extension of $B$, then also $B'\in\K_n$. 
This is clear for strong extensions of type 1. For strong extensions of type 2.
suppose that $(b,a,b')$ is a path with $a\in V_i,b,b'\in V_{i-1}$, and 
$\gamma=(a,b,\ldots , b',a)\subset B'$ is an $E_i\cup\ldots\cup E_{i-j}$-path of length $k$ containing
the new vertex $y$. Since the new vertex has exactly two neighbours $y_1,y_2$, this implies that the vertex is of type $V_m$ for some $i-j\leq m\leq i$ and $(y_1,y,y_2)$ is contained in $\gamma$.
By construction of strong extensions, there is some $z\in B$ such that $(y_1,z,y_2)$ is a path. Hence we may replace all occurrences of $y$ in $\gamma$ by $z$. Then $\gamma$ is contained in $B$ and we find the required path in $R(a)$ with all $V_{i-j}$-vertices occurring in $\gamma$. 
\end{proof}

This shows that the class $(\K_n,\leq)$ has a Hrushovski limit $M_n$, i.e. a countable $L_n$-structure $M_n$ whose strong subsets  are exactly the $L_n$-graphs in $\K_n$ and which is homogeneous for strong subsets: if $A,B\leq M_n$ then any isomorphism from $A$ to $B$ extends to an automorphism of $M_n$.
Here we say as usual that a  subset $A$ of $M_n$ is strong in $M_n$ if
$A\cap B\leq B$ for any finite set $B\subset M_n$.

\begin{proposition}\label{p:pseudospace}
The Hrushovski limit $M_n$ is a model of $T_n$.
\end{proposition}

\begin{proof}
By construction, $V_i\cup\ldots\cup V_{i+j}$ satisfies $(\Sigma 3)_j$ and $(\Sigma 4)_j$  for any $i,j$. In particular, $M_n$ satisfies 
$(\Sigma 3)_n$ and $(\Sigma 4)_n$.

\medskip
\noindent
{\bf $(\Sigma 1)_n$:}
In order to show that $M$ satisfies $(\Sigma 1)_n$, we first note that $V_i\cup V_{i+1}$ is a free pseudoplane for all $i=0,\ldots n-1$. Assume inductively that $V_j\cup\ldots \cup V_{j+i}$ 
is a free pseudospace of dimension $i$. To see that  $V_j\cup\ldots \cup V_{j+i+1}$ is a free pseudospace of dimension $i+1$, we need only verify $(\Sigma 2)_{i+1}$. Hence we have to show that  for $a\in V_j$  the residue $R(a)\cap (V_j\cup\ldots \cup V_{j+i+1})$ is a free pseudospace  of dimension $i$. We know by induction that $R(a)\cap (V_j\cup\ldots \cup V_{j+i})$ is a free pseudospace.

Clearly, 
\[R(a)\cap (V_{j+1}\cup\ldots \cup V_{j+i+1})=\bigcup \{R(b)\cap (V_{j+1}\cup\ldots \cup V_{j+i+1})\colon b\in V_{j+1}, E(a,b) \}.\] For each neighbour $b\in V_{j+1}$ of $a$, the set
$R(b)\cap (V_{j+1}\cup\ldots \cup V_{j+i+1})$ is a free pseudospaces of dimension $i-1$ by induction. Since $(V_{j+1}\cup\ldots \cup V_{j+i+1})$ is a free pseudospace of dimension $i$,  $(\Sigma 2)_{i+1}$ follows from the induction hypothesis.
Hence $V_0\cup\ldots \cup V_{n-1}$ and $V_1\cup\ldots \cup V_n$ are free pseudospaces of dimension $n-1$.

\medskip
\noindent
{\bf $(\Sigma 2)_n$:}
The proof of  $(\Sigma 2)_n$ is similar.
\end{proof}

We say that a model $M$ of $T_n$ is 
$\K_n$-saturated if for all finite $A\leq M$ and strong extensions $C$
of $A$ with $C\in\K_n$ there is a strong embedding of $C$ into $M$
fixing $A$ elementwise.
Clearly, by construction, $M_n$ is $\K_n$-saturated.

\begin{lemma}\label{l:K-sat}
A model $M$ of $T_n$ is $\omega$-saturated if and only if $M$ is $\K_n$-saturated.
\end{lemma}

\begin{proof}
Let $M$ be an $\omega$-saturated model of $T_n$. To show that $M$ is
  $\K_n$-saturated, let $A\leq M$ and $A\leq B\in\K_n$.  By
  induction we may assume that $F$ is an elementary strong extension of $A$
  and it is easy to see that $F$ can be imbedded over $A$ into $M$.
\end{proof}

\begin{corollary}
The theory $T_n$ is complete.
\end{corollary}

\begin{proof}
  Let $M$ be a model of $ T_n$. In order to show that $M$ is
  elementarily equivalent to $M_n$ choose an $\omega$-saturated
  $M'\equiv M$. By Lemma~\ref{l:K-sat}, $M'$ is
  $\K_n$-saturated. Now $M'$ and
  $M_n$ are partially isomorphic and therefore elementarily
  equivalent.
\end{proof}

We will see in Section~\ref{sec:building} that $T_n$ is the theory of
the building of type $A_{\infty, n+1}$ with infinite valencies.

\begin{definition}
 Following \emph{\cite{BP}} we call a subset $A$ of a model $M$ of $T_n$ \emph{nice} if 
\begin{enumerate}
 \item any $E_i$-path between elements of
$A$ lies entirely in $A$ and
\item if $a,b\in A$ are connected by a path in $M$ there is a path
from $a$ to $b$ inside $A$.
\end{enumerate}
\end{definition}
\begin{remark}\label{l:strong=nice}\upshape\label{r:qftp}
Note that a subset $A$ of $M_n$ is strong in $M_n$ if and only if it is nice. 
(This follows immediately from the definition of strong extension.)

Since $M_n$ is homogeneous for strong subsets, for any nice subset of $M_n$ the quantifier-free type determines the type, i.e. if $\bar a,\bar b$ are nice in $M_n$ and such that $\qftp(\bar a)=\qftp(\bar b)$, then $\tp(\bar a)=\tp(\bar b)$.
\end{remark}

We now work in a very saturated model $\overline{M}$ of $T_n$.

\begin{lemma}\label{l:strongext}
If $A$ is a finite set, there is a nice finite set $B$ containing
$A$.
\end{lemma}

\begin{proof}
Since single vertices are nice it suffices to prove the
following 

\medskip
\noindent
{\bf Claim:}\emph{ If $A$ is nice and $a$ arbitrary, then
there is a nice finite set $B$ containing $A\cup\{a\}$.}

\medskip
\noindent
\emph{Proof of Claim:} 
Of course we may assume $a\notin A$. If there is no path from $a$ to $A$, clearly $A\cup\{a\}$ is nice.
Hence we may also assume that there is some path $\gamma=(a=x_0,\ldots b)$ for some $b\in A$ and $\gamma\cap A=\{b\}$. It therefore suffices to prove the claim for
the case where $a$ has a neighbour in $A$. If $a$ has two neighbours $x,y\in A$ then
$(x,a,y)$ is a dense flag and $A\cup\{a\}$ is nice.

Now assume that $a\in V_i$ has a unique neighbour of type $V_{i+1}$ in $A$. 
(The other case then follows by self-duality.)
If the $E_i$-connected component of $a$  does not intersect $A$, then again  $A\cup\{a\}$ is nice.
Otherwise there is some $E_i$-path $\gamma=(x_0=a,\ldots x_m=b)$ in $M_n$ with $\gamma\cap A=\{b\}$. If for some $V_{i-1}$-vertex $x_k$ of $\gamma$
there is an $E_{i-1}$-path to some $c\in A$, then the $E_i$-path from $c$ to $b$ extends $(x_k,\ldots x_m=b)$ and is entirely contained in $A$ since $A$ is nice. Since $\gamma\cap A=\{b\}$, no such $x_k$ exists implying that $A\cup\gamma$ is nice.
\end{proof}

Let us say that $\gamma$ \emph{changes
direction in $x_i$} if $x_i\in V_j$ and either $x_{i-1},x_{i+1}\in V_{j-1}$ or $x_{i-1},x_{i+1}\in V_{j-1}$ for some $j$.  Clearly a path which doesn't change direction is a dense flag.

We will use the following  easy consequence of $(\Sigma 4)_{j\leq n}$.

\begin{lemma}\label{l:cycle}
If $\gamma=(a,b,\ldots, b',a)$ is a simple cycle,
then any two vertices $x,y\in \gamma\cap (V_{j-1}\cup V_j)$ are joined by an $E_j$-path.
\end{lemma}

\begin{proof}
Suppose that $j,j+m$ are minimal and maximal, respectively, with 
$\gamma\cap V_j, V_{j+m}\neq\emptyset$.
Apply $(\Sigma 4)_m$ to all vertices of $\gamma\cap V_{j+m}$ and replace any path of
the form $(x,y,z)$ with $y\in V_{j+m}$ by a path from $x$ to $z$ contained in
$V_j\cup\ldots \cup V_{j+m-1}$. This yields a new path contained in $V_j\cup\ldots \cup V_{j+m-1}$. 
Next apply $(\Sigma 4)_{m-1}$ to each vertex of type $V_j$ to obtain a path in
$V_{j+1}\cup\ldots \cup V_{j+m-1}$. Alternating between replacing the top and bottom extremal peaks by a path using fewer levels, we end up with an $E_k$-path for some $j\leq k\leq j+m-1$ between vertices of $\gamma$. Since an $E_k$-path changes direction
at every vertex, we can apply the same procedure in order to replace this path
by an $E_{k'}$-path for any $j\leq k'\leq j+m-1$.
\end{proof}

\begin{definition}\label{def:reduced}
We call a path $\gamma=(x_0,\ldots x_k)\subseteq V_j\cup\ldots \cup V_{j+m}$ \emph{reduced} if $\gamma$ is a flag or if the following
holds:
\begin{enumerate}
\item if $m=1$ the path $\gamma$ is reduced if it does not contain any repetition.
\item for any simple cycle $\gamma'$ containing $(x_i,\ldots, x_{i+q})$
replacing all $(y,x,z)$ with  $x\in (x_{i+1},\ldots, x_{i+q-1})\cap V_{j+m}$ by an $E_{j+m-1}$-path in $R(x)$ yields a reduced path  $(x_i,\ldots, y,\ldots z,\ldots, x_{i+q})\subseteq V_j\cup\ldots \cup V_{j+m-1}$
\end{enumerate}
\end{definition}

Note that the definition and Lemma~\ref{l:cycle} immediately imply the following:
\begin{remark}\upshape\label{r:reducedpath}
Suppose that every reduced path from $a$ to $b$ contains $x$ and let $\gamma_1,\gamma_2$ be paths from $a$ to $x$ and from $x$ to $b$ respectively. Then the path
$\gamma_1\gamma_2$ is reduced if and only if $\gamma_1$ and $\gamma_2$ are.

If $\gamma=(a,\ldots, b)$ is a reduced path and changes direction in $x\in V_k$, then either every reduced path from $a$ to $b$ contains $x$ or the path $(y,x,z)\subseteq \gamma$ can be replaced by an $E_{k-1}$-path (or by an $E_k$-path) $(y,\ldots z)$
such that $(a,\ldots, y,\ldots, z,\ldots b)$ is still reduced.
\end{remark}

\begin{lemma}\label{l:nicereplacement}
Suppose $\gamma=(a=x_0,\ldots, x_s=b)$ is a reduced path from $a$ to $b$ which can be completed into a simple cycle $(a,x_1,\ldots, b,\ldots,a)$. If $\gamma$ changes direction in $x\in V_k$, then there is a reduced path $\gamma'=(a=y_0,\ldots, y_t=b)$ not containing $x$ such that $x$ is on the $E_{k}$-(or $E_{k+1}$-) path between appropriate elements of $\gamma'$.
\end{lemma}
\begin{proof}
To fix notation let us assume that the neighbours of $x$ in $\gamma$ are $y,z\in V_{k-1}$. By Lemma~\ref{l:cycle} we may replace the path $(y,x,z)$  by an $E_k$-path. This path is obviously
still reduced and $x$ is on the $E_{k+1}$-path from $y$ to $z$.
\end{proof}

\begin{corollary}\label{l:nicepath}
Suppose $D$ is a nice set containing an $E_j$-path from $a$ to $b$. Suppose furthermore that  $\gamma=(a=y_0,\ldots, y_t=b)$ is a reduced path such 
that the composition with the $E_j$-path is a simple cycle.  
If $\gamma$ changes direction in $c$, then $c\in D$.
\end{corollary}

Using the fact that $M_n$ is $\omega$-saturated and homogeneous for strong substructures we can now describe the algebraic
closure:

\begin{lemma}\label{l:acl2}
The algebraic closure $\cl(a,b)$ is the intersection of all nice sets
containing $a,b$. 

\end{lemma}
\begin{proof}
Since there are finite nice sets containing  $a,b$, the  intersection of all nice sets containing $a,b$ is
certainly finite and invariant over  $a,b$, hence contained in $\cl(a,b)$.

Conversely let $D$ be a nice set containing $\{a,b\}$ and $c\notin D$.
If there is no reduced path between $a$ and $b$ changing direction in $c$, then $c$ has infinitely many conjugates over $ab$, so $c\notin\cl(ab)$. 

So suppose there is a reduced path $\gamma'$ from $a$ to $b$ changing direction in $c$. Let $\gamma\subset D$ be a path from $a$ to $b$. 
 Composing these paths we obtain a simple cycle $\gamma"$ through $c$ consisting of subpaths $\gamma_1'\subseteq \gamma_1$ of $\gamma'\subseteq\gamma$, respectively, from some $a'$ to $b'$.  In particular $\gamma_1$ is not a flag.
We may assume inductively that $\gamma_1'\cap D=\{a',b'\}$. If $a'\in V_j$, then possibly after exchanging the roles of $a',b'$, there is some $x\in\gamma\cap (V_{j-1}\cup V_j)$ closest to $b'$. Then $a',x$ are connected  by an $E_j$-path entirely contained in $D$ and $(x,\ldots b')$ is a flag. Since the $E_j$-path changes direction in every vertex, we also find an $E_{j-1}$-path inside $D$ to the next vertex of the flag $(x,\ldots b')$.
Since $\gamma_1'$ doesn't meet $D$ except in $a',b'$, we see that ${\gamma'_1}^\smallfrown (b',\ldots, x)$ is a reduced path. Now Lemma~\ref{l:nicepath} implies $c\in D$, a contradiction. Hence $c\notin\cl(ab)$.
\end{proof}

Note that the proof shows in fact the following:

\begin{corollary}\label{r:flagcl}\upshape
A vertex $x\neq a,b$ is in $\cl(ab)$ if and only if there is a reduced path
from $a$ to $b$ that changes direction in $x$. Hence $\cl(ab)=\{a,b\}$ if and only if $a,b$ is a flag or $a$ and $b$ are not connected. In fact, we have $\dcl(ab)=\acl(ab)$.
\end{corollary}

\begin{lemma}\label{l:treeacl}
If $g\in\cl(A)$, there exist $a,b\in A$ with $x\in\cl(ab)$.

\end{lemma}

\begin{proof}
We may assume that $A$ is finite. By induction it suffices to prove 
that if $d\in\cl(bc),g\in\cl(ad)$, then $g\in\cl(ab)\cup\cl(bc)\cup\cl(ac)$.

By Corollary~\ref{r:flagcl} there is a reduced path $\gamma_1=(b,\ldots,d,\ldots, c)$
changing direction in $d$ and a reduced path $\gamma_2=(d,\ldots ,g,\ldots, a)$ changing
direction in $g$. 
If $\gamma_1\cup\gamma_2\in V_i\cup V_{i+1}$ for some $i$, then clearly either
$(a,\ldots,g\ldots,d,\ldots,b)$ or
$(a,\ldots,g\ldots,d,\ldots,c)$ is reduced.

Now assume that $\gamma_1\cup\gamma_2\in V_i\cup\ldots\cup V_{i+k}$.
Clearly we may  assume
that $d$ is not contained in every reduced path from $a$ to $b$ or in
every reduced path from $a$ to $c$. Furthermore, we may reduce to the case where no vertex of $\gamma_1$ is contained
in every reduced path from $b$ to $c$ and no vertex of $\gamma_2$  is contained
in every reduced path from $a$ to $d$.

By Remark~\ref{r:reducedpath} we may therefore replace any element $x\in V_{i+k}$ in the interior of $\gamma_1$ and $\gamma_2$ by an $E_{i+k-1}$ path between its neighbours. If $d\in V_{i+k}$, we may replace $d$ as follows: the neighbour $w$ of $d$ in $\gamma_2$ either appears in the $E_{i+k-1}$ path between the neighbours of $d$ in $\gamma_1$ or that path extends to $w$ on one of its ends. We may thus
replace the paths $(b,\ldots,d,\ldots, c)$ and $(a,\ldots g,\ldots, d)$ by  reduced paths
$(b,b_1,\ldots,w, \ldots,c_1,c)$ and $(a,a_1,\ldots w)$ contained in $V_i\cup\ldots\cup V_{i+k-1}$,
 except possibly for the endpoints. We may assume that at most $a\in V_{i+k}$
 since we may exchange the roles of $V_i$ and $V_{i+k}$ and replace the elements of $V_i$ instead.
  
By induction assumption, at least one of $(a_1,\ldots w,\ldots b_1)$ and  $(a_1,\ldots w,\ldots c_1)$
is reduced, and replacing the new $E_{i+k-1}$-paths by the old $E_{i+k}$-paths,
this path remains reduced and changes direction in $g$. If $a,b\in V_{i+k}$,
then we may exchange the roles of $V_i$ and $V_{i+k}$ and replace the elements of $V_i$ instead. Thus we see that the path remains reduced when adding the endpoint.
\end{proof}

\begin{proposition}\label{r:acl}
For any set $A$, the algebraic closure $\cl(A)$ is the intersection of all nice sets containing it.
\end{proposition}

\begin{proof}
Clearly, we may restrict ourselves to finite sets $A$.
As before we see that the intersection of all nice sets containing $A$ is contained in $\cl(A)$.

For the converse, assume that $c$ is not in the intersection of all nice
sets containing $A$. If $c$ was in $\cl(A)$, then by Lemma~\ref{l:treeacl} there are $a,b\in A$ with $c\in\cl(ab)$. Then by Lemma~\ref{l:acl2} $c\in\bigcap_{a,b\in D}D\subset\bigcap_{A\subseteq D}D$, a contradiction.
\end{proof}

\begin{proposition}
For any vertex $a$ and set $A$, there is a flag $C\in\cl(A)$ such that for  any $b\in\cl(A)$ there is a reduced path from $a$ to $b$ passing through one
of the elements of $C$. 
\end{proposition}

The flag $C$ is called the \emph{projection} from $a$ to $A$ and we write $C=\proj(a/A)$.
Note that  $\proj(a/A)=\emptyset$ if and only if $a$ is not connected to any vertex of $\cl(A)$.
\begin{proof}
Let $b_1,b_2\in\cl (A)$ and let $\gamma_1=(x_0=a,\ldots ,x_m=b_1)$ and $\gamma_2=(y_0=a,\ldots ,y_l=b_2)$ be reduced paths such that $i,j$ are minimal possible with $x_i,y_j\in\cl(A)$.
By composing the initial segments of $\gamma_1$ and $\gamma_2$ and reducing we obtain a reduced path from $x_i$ to $y_j$ intersecting $\cl(A)$ only in $x_i,y_j$ since $x_i, y_j$ were chosen at minimal distance from $a$. By Corollary~\ref{r:flagcl} $x_i,y_j$ is a flag. Thus the set of such vertices forms a flag $C$. 
\end{proof}

It is now easy to show the following:

\begin{theorem}\label{T:stable}
The theory $T_n$ is $\omega$-stable.

\end{theorem}

\begin{proof}
Let $M$ be a countable model and let $\bar d$ be a tuple from  $\overline{M}$. Let $C\in M$ be the finite set of projections from
$\bar d$ to $M$. Then the type $\tp(\bar d/M)$ is determined by $\tp(\bar d/C)$.
By Lemmas~\ref{l:strong=nice} and \ref{l:strongext}, $\bar d\cup C$ is contained in a finite strong subset of $M_n$ and for
such subsets the quantifier-free type determines the type by Remark~\ref{r:qftp}. Hence there are only countably many types over a countable model.
\end{proof}

In fact, it is easy to see directly without counting types that $T_n$ is superstable (see Remark~\ref{r:superstable}).

\begin{corollary}
The \fps has weak elimination of imaginaries.
\end{corollary}

\begin{proof}
Let $a$ be a vertex and $A$ any set. Then we can choose $\Cb(stp(a/A))$
as the projection of $a$ on $A$. This is a finite set.
\end{proof}

The following immediate corollary will be very useful:
\begin{corollary}\label{c:forking}
The vertex $a$ is independent from $A$ over $C$ if $\proj(a/AC)\subseteq\cl(C)$. In particular, $a$ is independent from $A$ over  $\emptyset$ if and only if $a$ is not connected to any vertex of $\cl(A)$.
\end{corollary}

\begin{sideremark}\upshape\label{r:superstable}
As in \cite{TZU} we could have defined a notion of independence on models of $T_n$ by saying
\[\indep{A}{C}{B}\]
if and only if $\proj(a/BC)\subseteq\cl(C)$ for all $a\in\cl(A)$.
It is easy to see that this notion of independece satisfies the characterizing properties of forking in stable theories (see \cite{TZ} Ch. 8) and hence agrees with the usual one. Note that the existence of nonforking extensions follows
from the construction of $M_n$ as a Hrushovski limit.
Since we have just seen that for any type $\tp(a/A)$ there is a finite set $A_0$
such that $\indep{a}{A_0}{A}$ this shows directly (without counting types) that $T_n$ is superstable.
\end{sideremark}

Using this description of forking it is easy to give a list of regular types
such that any nonalgebraic type is non-orthogonal to one of these. This is entirely
similar to the list given in \cite{BP} and we omit the details but will return to this point in Section~\ref{sec:types}. It is also clear from
this description of forking that the geometry on these types is trivial.

\section{Ampleness}
We now recall the definition of a theory being $n$-ample:

\begin{definition}
A theory $T$ eliminating imaginaries is called $n$-ample if possibly after naming parameters there are tuples $a_0,\ldots a_n$ in $M$ such that the following holds:

\begin{enumerate}
\item for $i=0,\ldots n-1$ we have 
\[\cl(a_0,.\ldots a_{i-1},a_i)\cap \cl(a_0,.\ldots a_{i-1},a_{i+1})=\cl(a_0,.\ldots a_{i-1});\]
\item $\notind{a_n}{}{a_0},$ and
\item $\indep{a_n}{a_i}{a_0\ldots a_i}$ \ for  $i=0,\ldots n-1$.
\end{enumerate}

\end{definition}

\begin{theorem}\label{T:n-ample}
The theory $T_n$ is $n$-ample and 
any maximal flag $(x_0,\ldots x_n)$ in $M_n$ is a witness for this.
\end{theorem}

\begin{proof}
This follows immediately from the description of $\cl$ in Lemma~\ref{r:acl} and
of forking in Corollary~\ref{c:forking}.
\end{proof}

\begin{theorem}\label{T:notample}
The \fpsn is not $n+1$-ample.
\end{theorem}

\begin{proof}
Suppose towards a contradiction that $a_0,\ldots a_{n+1}$ are witnesses for $T_n$ being $n+1$-ample over some set of parameters $A$. We have 
\[\notind{a_{n+1}}{A}{a_0},\]
\[\indep{a_{n+1}}{Aa_i}{a_0\ldots a_i},i=0,\ldots n.\]

By the first condition there are vertices in $\cl(a_0)$ and in $ \cl(a_{n+1})$ which are in the same connected component. Put $f_0=\proj(a_{n+1}/a_0A)\in\cl(a_0A)$ and $f_{n+1}=\proj(f_0/a_{n+1}A)\in\cl(a_{n+1}A)$.

Since \[\indep{a_{n+1}}{Aa_i}{a_0\ldots a_i}, i=1,\ldots n\] using Corollary~\ref{c:forking} we inductively find 
flags 
\[f_i=\proj(f_{n+1}/f_0f_1\ldots f_{i-1}a_iA)=\proj(f_{n+1}/a_iA)\in\cl(a_iA),i=1,\ldots n\] such that \[\indep{f_{n+1}}{f_i}{f_0f_1\ldots f_i}.\]

For $i=1,\ldots n$ we clearly have
\[\cl(f_0,f_1.\ldots f_{i-1},f_i)\cap \cl(f_0,.\ldots f_{i-1},f_{i+1})\subseteq \cl(a_0,.\ldots a_{i-1}A).\]

By construction, there is a reduced path $\gamma=(f_0,x_1,\ldots x_k=f_{n+1})$ containing a vertex
of each of the $f_i$ in ascending order. Since we cannot have a flag containing more than $n$ elements, there must be some vertex $x$ in $\gamma$ where $\gamma$ changes direction.
For some $i$ we then have  $x\in f_{i+1}$ or $x$ occurs in $\gamma$ between an element of $f_i$ and an element of $f_{i+1}$. By Corollary~\ref{r:flagcl} we have
$x\in\cl(f_if_{i+1})\cap\cl(f_if_{i+2})$. Then
\[\indep{x}{f_i}{a_0a_1,\ldots a_iA},\] so $x\notin\cl(a_0a_1,\ldots a_iA)$, 
a contradiction.
\end{proof}

 The proof shows that in fact the following stronger ampleness result holds:
\begin{corollary} If $a_0,\ldots a_n$ are witnesses for $T_n$ being $n$-ample, then there are vertices $b_i\in\cl(a_i)$ such that $(b_0,\ldots b_n)$ is a flag.
\end{corollary}

\section{Buildings and the prime model of $T_n$}\label{sec:building}

We now turn towards constructing the prime model $M_n^0$ of $T_n$ as a Hrushovski-limit.
We will show that $M_n^0$ is the building associated to a right-angled Coxeter group.

For this purpose we now consider an expansion $L_n'$ of the language $L_n$ by binary
function symbols $f^i_k$. 
For an $L_n$-graph $A$ we put  $f^i_k(x,y)=z$ if $z$ is the $k^{th}$ element on a unique shortest $E_i$-path of length at least $k$ from $x$ to $y$ and $z=x$ otherwise.

We say that an $L_n$-graph $A$ is $E_i$-connected if the set $V_{i-1}(A)\cup V_i(A)$ is connected.

\begin{definition}
Let $\K'$ be the class of finite $L_n'$-graphs $A\in\K$ which are  $E_i$-connected for $i=1,\ldots n$ and additionally satisfy the following condition:
\begin{itemize}

\item[6.] If $a\in A$ is of type $V_j$, then the residue $R(a)$  is $E_i$-connected for $i=1,\ldots n$.

\end{itemize}

\end{definition}

Note that $\K'$ is closed under finitely generated substructures by the choice of language.

\begin{definition} Let $A$ be a finite $L_n'$-graph which is 
$E_i$-connected for
$i=1,\ldots n$. The following extensions are called elementary 
strong extensions of $A$: 
\begin{enumerate} 
\item add a vertex of type $P$ or $E$ to $A$ which is
connected to at most one vertex of $A$ and such that the extension is still
$E_i$-connected for all  $i=1,\ldots n$. 
\item If $(x,y,z)$ is a dense flag in
$A$, add a vertex $y'$ of the same type as $z$ to $A$ such that $(x,y',z)$ is a
flag. 
\item if $|A|\leq 1$ contains no line, add a vertex of appropriate type
which is connected to the vertex of $A$ if $A\neq\emptyset$. 
\end{enumerate}

Again we write $A\leq B$ if  $B$ arises from $A$  by finitely many elementary strong extensions.
\end{definition}

We next show that $(\K_n',\leq)$ has the amalgamation property for $\leq$-extensions.

\begin{lemma} If $A$ contains a flag of type $(V_1,\ldots V_{n-1})$ and 
$A\leq B,C$ are in $\K_n'$, then $D:=B\otimes_AC\in\K_n'$ and $B,C\leq D$.
\end{lemma}

\begin{proof}

Clearly, $B,C\leq D$ and $D$ is $E_i$-connected for all $i=1,\ldots n$ since $A$ contains a flag. 
To see that $D\in\K$, note that if $B\in \K$ and $B'$ is an
elementary strong extension of $B$, then also $B'\in\K$.
\end{proof}

This shows that the class $(\K,\leq)$ has a Hrushovski limit $M^0_n$.
Clearly, $M^0_n$ is $E_i$-connected for $i=1,\ldots n$ 
and since any two vertices of $M_n^0$ lie in a
maximal flag, it follows that $M_n^0$ is in fact $n$-connected. Note that an $L_n'$-substructure of $M_n^0$ is 
automatically nice, see Remark~\ref{r:qftp}.

The same proof as in the case of $M_n$ shows the first part of
the following proposition:

\begin{proposition}\label{p:primmodell}
The Hrushovski limit $M^0_n$ is a model of $T_n$ and
 $M^0_n$ is the unique countable model
of $T_n$ which is $E_i$-connected for $i=1,\ldots n$ and
such that every vertex is contained in a maximal flag.
\end{proposition}
 
(Note that in \cite{BP} the corresponding Remark 3.6 of uses Lemma 3.2, 
which is not correct as phrased there: $M_n^0$ and $M_n^0\cup\{a\}$ with $a$ an isolated point are not isomorphic, but satisfy the assumptions of Remark 3.6.)

The uniqueness part of Proposition~\ref{p:primmodell} follows directly from the following  theorem and Proposition 5.1 of \cite{HP} which states that this type of building is uniquely determined by its associated Coxeter group and the 
cardinality of the residues.

\begin{theorem}\label{t:building}
$M^0_n$ is a building of type $A_{\infty,n+1}$ all of whose residues have cardinality $\aleph_0$.
\end{theorem}

Recall the following definitions (see e.g. \cite{G}).

Let $W$ be the Coxeter group

\[W=\langle t_0,\ldots t_n\colon t_i^2=(t_it_k)^2=1, i,k=0\ldots n, |k-i|\geq 2 \rangle,\]
 whose associated diagram we call $A_{\infty,n+1}$.

\begin{definition}\label{d:building}
A building of type  $A_{\infty,n+1}$ is a set $\Delta$ with a \emph{Weyl distance function}
$\delta: \Delta^2\to W$ such that the following two axioms hold:
\begin{enumerate}
 \item For each $s\in S:=\{t_i,i=0,\ldots n\}$, the relation $x\sim_s y$ defined by
$\delta(x,y)\in\{1,s\}$ is an equivalence relation on $\Delta$ and each equivalence class
of $\sim_s$ has at least $2$ elements.
\item Let $w=r_1r_2\ldots r_k$ be a shortest representation of $w\in W$ with $r_i\in S$
and let $x,y\in\Delta$. Then $\delta(x,y)=w$ if and only if there exists a sequence
of elements $x,x_0,x_1,\ldots, x_k=y$ in $\Delta$ with $x_{i-1}\neq x_i$ and
$\delta(x_{i-1},x_i)=r_i$ for $i=1,\ldots, k$. 
\end{enumerate}

\end{definition}

A sequence as in 2. is called a \emph{gallery of type} $(r_1,r_2,\ldots, r_k)$.
The gallery is called reduced if the word $w=r_1r_2,\ldots, r_k$ is reduced, i.e.
a shortest representation of $w$.

\medskip

We now show how to consider $M_n^0$ as a building of type  $A_{\infty,n+1}$.

\begin{proof} (of Theorem~\ref{t:building})
We extend the set of edges of the $n+1$-coloured graph $M_n^0$ by putting edges between any
two vertices that are incident in the sense of Definition~\ref{d:incidence}\ref{d:incidence1}.
In this way, flags of $M_n^0$ correspond to a complete subgraph of this extended
graph, which thus forms a simplicial complex. A maximal simplex consists of $n+1$ vertices each of a different type $V_i$. (Such a simplex is called a \emph{chamber}.) Let $\Delta$ be the set of maximal simplices in this graph.
Define $\delta:\Delta^2\to W$ as follows:

Put $\delta(x,y)=t_i$ if and only if the flags $x$ and $y$ differ exactly in the vertex of type $V_i$. Extend this  by putting
$\delta(x,y)=w$ for a reduced word $w=r_1r_2\ldots r_k$ if and only if there exists a sequence
of elements $x=x_0,x_1,\ldots, x_k=y$ in $\Delta$ with $x_{i-1}\neq x_i$ and
$\delta(x_{i-1},x_i)=r_i$ for $i=1,\ldots, k$.

Clearly, with this definition of $\delta$, the set $\Delta$ satisfies the first
condition of Definition~\ref{d:building}. In fact, for all $s\in S$ every equivalence class $\sim_s$ has cardinality $\aleph_0$.

We still need to show that $\delta$ is well-defined, i.e. if we have to show the following for any $x,y\in\Delta$: 
if there are reduced galleries $x_0=x,x_1,\ldots, x_k=y$ and $y_0=x,y_1,\ldots, y_m=y$ of type $(r_1,r_2,\ldots,r_k)$ and $(s_1,\ldots s_m)$, respectively,
then in $W$ we have $r_1r_2\ldots r_k=s_1\ldots s_m$.
Equivalentely, we will show the following, which completes the proof of Theorem~\ref{t:building}:

\medskip
\noindent
{\bf Claim:} There is no reduced gallery $a_0,a_1,\ldots, a_k=a_0$ for $k>0$ in $M_n^0$.

\medskip
\noindent
\emph{Proof of Claim.} 
Suppose otherwise. Let  $a_0,a_1,\ldots, a_k=a_0$ be a reduced gallery of type $(r_1,\ldots r_k)$ for some $k>0$. Note that the flags $a_i$ and $a_{i+1}$ contain the same vertex of
type $V_j$ as long as $r_i\neq t_j$. 

Now consider the sequence of vertices of type $V_n$ and $V_{n-1}$ occurring in this gallery. Since $V_n\cup V_{n-1}$ contains no cycles, the sequence of vertices of type $V_n$ and $V_{n-1}$ occurring in this gallery will be of the
form 
\begin{equation}\label{eq:gallery}
  (x_1,y_1,x_2,y_2,\ldots ,x_i,y_i,x_i,y_{i-1},\ldots y_1,x_1)
\end{equation} 
  with $x_i\in V_n,y_i\in V_{n-1}$ and $x_i$ a neighbour of $y_i$ and $y_{i-1}$ in the original graph.
This implies that at some place in the gallery type there are two occurrences of $t_n$ which are not separated by an occurrence of $t_{n-1}$ (or conversely). Since $t_n$ commutes with all $t_i$ for $i\neq n-1$ and the word $r_1\ldots r_k$ is reduced, there are two occurrences of $t_{n-1}$ which are not separated by an occurrence of $t_n$, say $r_j,r_{j+m}=t_{n-1}$ with $r_{j+1},\ldots, r_{j+m-1}\neq t_n$. 

We now consider the gallery $a_j,\ldots a_{j+m}$ of type $(r_j=t_{n-1},r_{j+1},\ldots, r_{j+m}=t_{n-1})$. Notice that by (\ref{eq:gallery}), the flags $a_j$ and $a_{j+m}$ have the same $V_n$ and the
same $V_{n-1}$ vertex. Since $M_n^0$ does not contain any $E_{n-1}$-cycles, 
the sequence of $V_{n-1}$- and $V_{n-2}$-vertices
appearing in this sequence must again be of the same form as in (\ref{eq:gallery}).
Exactly as before we find two occurrences\footnote{If $t_{n-2}$ does not occur in the type of the gallery, this would contradict the assumption that the type is reduced since $t_{n-1}$ commutes with all $t_i$ for $i\neq n,n-2$.} of $t_{n-2}$ in the gallery type of $a_j,\ldots a_{j+m}$ which are not separated by an occurrence of $t_{n-1}$.
 Continuiung in this way, we eventually find two occurrences of $t_1$ which are not separated by any $t_i$. Since $t_1^2=1$ this contradicts the assumption that the gallery be reduced.

\end{proof}

The proof shows in fact the following:
\begin{corollary}
A model of $T_n$ is a building of type $A_{\infty,n+1}$ if and only if
it is $E_i$-connected for all $i$ and every vertex is contained in a maximal flag.
\end{corollary}

\begin{theorem}
The building $M_n^0$ is the prime model of $T_n$

\end{theorem}

\begin{proof}
To see that $M_n^0$ is the prime model of $T_n$
note that  for any flags $C_1,C_2\in M_n^0$  and gallery $C_1,=x_0,\ldots ,x_k=C_2$ the set of vertices occuring in this gallery is $E_i$-connected 
for all $i$. Hence by Remark~\ref{r:qftp} its type is determined by the
quantifier-free type.

Thus, given a maximal flag $M$ in any model of $T_n$ and  a maximal  flag  $c_0$ of $M_n^0$ we can embed $M_n^0$ into $M$ by moving along the galleries of $M_n^0$.
\end{proof}

\section{Ranks and types}\label{sec:types}

Recall that for vertices $x,y\in M_n^0$ with $x\in V_i,y\in V_j$ the \emph{Weyl-distance} $\delta(x,y)$ equals $w\in W$ if there are flags $C_1,C_2$ containing $x,y$, respectively, with 
$\delta(C_1,C_2)=w'$ and such that $w$ is the shortest representative of the
double coset $\langle t_k\colon k\neq i\rangle^W w'\langle t_k\colon k\neq j\rangle^W$ (where  as usual $\langle X\rangle^W$ denotes the subgroup of $W$ generated by $X$). 

The following is clear:
\begin{proposition}\label{p:qe}
The theory $T_n$ has quantifier elimination in a language containing predicates
$\delta_w^{i,j}$ for Weyl distances between vertices of type $V_i$ and of type $V_j$.
\end{proposition}

For any small set $A$ in a large saturated model we have the following kinds of regular types: 

\begin{enumerate}
\item[(I)] $\tp(a/A)$ where $a\in V_i$ is not connected to any element in $\cl(A)$
\item[(II)]$\tp(a/A)$ where $a\in V_i$ is incident with some $b\in\cl(A)\cap V_j$ but not connected in $R(b)$ to any vertex in $\cl(A)\cap R(b)$.
\item[(III)] $\tp(a/A)$ where $a\in V_i$ is incident with some $x,y\in \cl(A)$ such that $(x,a,y)$ is a flag with $x\in V_k,y\in V_j$; and as a special case of this we have
\item[(IV)] $\tp(a/A)$ where $a\in V_i$ has neighbours $x,y\in \cl(A)$ such that $(x,a,y)$ is a (necessarily dense) flag.
\end{enumerate}

By quantifier elimination any of these descriptions determines a complete type.
Using the description of forking in Corollary~\ref{c:forking} one sees easily that each of these types is regular and trivial.

Clearly, any type in (IV) has $U$-rank $1$ and in fact Morley rank $1$ by quantifier elimination. It also follows easily that $\MR(a/A)=\omega^n$ if
$\tp(a/A)$ is as in (I). In case (II) we find that  $\MR(a/A)=\omega^{n-j-1}$ 
or $\MR(a/A)=\omega^{j-1}$ depending on whether or not $i<j$.
In case (II) we have  $\MR(a/A)=\omega^{|k-j|-2}$.

Just as in \cite{BP} we obtain:
\begin{lemma}
Any regular type in $T_n$ is non-orthogonal to a type as in (I) or as in (IV).
\end{lemma}
\begin{proof}
Let $p=\tp(b/\cl(B))$. If $b$ is not connected to $\cl(B)$, then $p$ is
as in (I), so we may assume that $\proj(b/B)=C\neq\emptyset$. Let $a$ be 
a vertex on a short path from $b$ to $C$ incident with an element of $C$.
Then by Corollary~\ref{c:forking} we see that $p$ is non-orthogonal to
$\tp(a/C)$.
\end{proof}

\noindent
{\bf Acknowledgement:} I would like to thank Linus Kramer, in particular for providing reference \cite{HP}.

\vspace{.5cm}
\noindent\parbox[t]{15em}{
Katrin Tent,\\
Mathematisches Institut,\\
Universit\"at M\"unster,\\
Einsteinstrasse 62,\\
D-48149 M\"unster,\\
Germany,\\
{\tt tent@math.uni-muenster.de}}

\end{document}